\documentclass[12pt]{article}

\usepackage{amsmath,amssymb}
\usepackage{tikz}
\usepackage{color}
\usepackage[toc]{appendix}
\usepackage{graphicx}
\usepackage{fancyhdr}
\usepackage{enumitem}
\usepackage{bbm}
\usepackage{graphicx}
\usepackage{parskip}
\usepackage{float}
\usepackage{chngpage}
\usepackage{calc}
\usepackage{bigints}
\usepackage{tikz}
\usepackage{array}
\usepackage{bbm}
\usepackage{booktabs}
\usepackage{rotating}
\usepackage{multirow}
\usepackage{adjustbox}
\usepackage{tabularx}
\usepackage{enumitem}
\usepackage{verbatim}
\usepackage{mathtools}
\usepackage{ragged2e}
\usepackage[makeroom]{cancel}
\usepackage{enumitem}
\usepackage{caption}
\usepackage{mathtools}
\usepackage{bbm}
\usepackage{amsmath,amsfonts,amsthm,bm}
\usepackage{amsmath,amssymb}
\usepackage{appendix}
\usepackage{graphicx}
\usepackage{tikz}
\usepackage{verbatim}
\usepackage{amsthm}
\usepackage[english]{babel}
\usepackage{hyphenat}
\usepackage[makeindex]{imakeidx}
\usepackage{tikz}

\usepackage{mathabx}
\usetikzlibrary{datavisualization}
\usetikzlibrary{matrix}
\usetikzlibrary{datavisualization.formats.functions}
\usetikzlibrary{arrows.meta}

\usetikzlibrary{positioning}

\newtheorem{theorem}{Theorem}[section]
\newtheorem{proposition}[theorem]{Proposition}

\newtheorem{definition}[theorem]{Definition}

\newtheorem{remark}[theorem]{Remark}

\usepackage[colorlinks=true, citecolor=blue, urlcolor=blue]{hyperref}

\usepackage[
natbib=true,
citestyle=authoryear,
bibstyle=numeric,
]{biblatex}
\makeatletter
\def\section{\@startsection {section}{1}{\z@}{3.25ex plus 1ex minus
		.2ex}{1.5ex plus .2ex}{\large\bf}}
\def\subsection{\@startsection{subsection}{2}{\z@}{3.25ex plus 1ex minus
		.2ex}{1.5ex plus .2ex}{\normalsize\bf}}
\@addtoreset{equation}{section} 
\makeatother
\makeatletter
\newsavebox{\@brx}
\newcommand{\llangle}[1][]{\savebox{\@brx}{\(\m@th{#1\langle}\)}%
	\mathopen{\copy\@brx\kern-0.5\wd\@brx\usebox{\@brx}}}
\newcommand{\rrangle}[1][]{\savebox{\@brx}{\(\m@th{#1\rangle}\)}%
	\mathclose{\copy\@brx\kern-0.5\wd\@brx\usebox{\@brx}}}
\makeatother

\title{Feynman-Kac formula for the heat equation driven by time-homogeneous white noise potential. }

\author{ Ramiro Scorolli\thanks{Dipartimento di Scienze Statistiche Paolo Fortunati, Università di Bologna, Bologna, Italy. \textbf{e-mail}: ramiro.scorolli2@unibo.it}}

\date{\today}

\begin{document}
	
	\maketitle
	
	\bigskip

\begin{abstract}
	We present a Feynman-Kac formula for the $1$-dimensional stochastic heat equation (SHE) driven by a time-homogeneous Gaussian white noise potential, where the noise is interpreted in the Wick-It\^o-Skorokhod sense.  
	Our approach consist in constructing a Wong-Zakai-type approximation for the SHE from which we are able to obtain an ``approximating Feynman-Kac'' representation via the reduction of the approximated SHE to a deterministic partial differential equation (PDE). Then we will show that those ``approximating Feynman-Kac'' converge to a well defined object we will call ``formal Feynman-Kac'' representation which happens to coincide with the unique solution of SHE.
\end{abstract}



\bigskip	
	
\section{Introduction.}
	
In this work we will deal with the $1$-dimensional stochastic heat equation
	\begin{align}\label{SHE}
		\begin{cases}
			\partial_t u(t,x)=\frac 1 2 \partial_{xx}^2 u(t,x)+u(t,x)\diamond \frac{d}{dx}W_x,\; (t,x)\in [0,T]\times\mathbb R\\
			u(0,x)=u_0(x),
		\end{cases}
	\end{align}
driven by the (distributional) derivative of the Brownian motion  $\{W_x\}_{x\in\mathbb R}$. From now on, in a slight abuse of notation we will denote $\dot W_x=\frac{d}{dx}W_x$.

The initial condition $u_0:\mathbb R\to\mathbb R$ is assumed to be a bounded, deterministic Borel-measurable function.
In (\ref{SHE}) the symbol ``$\diamond$'' denotes the Wick product (e.g. \cite{holden2009stochastic} or \cite{gjessing2020wick} ) and implies that the noise is interpreted in the It\^o-Skorohod sense.

In the case of space-time white noise potential equation (\ref{SHE}) has been extensively studied (e.g. \cite{potthoff1998generalized},\cite{hu2009stochastic},\cite{hu2011feynman},\cite{bertini1995stochastic} and references therein). On the other hand if the noise is assumed to be space-homogeneous white noise the equation (\ref{SHE}) is just a particular case of the well known Zakai equation from non-linear filtering theory (see for instance the original paper \cite{zakai1969optimal} or the review \cite{heunis1990stochastic} and references therein).

Nevertheless the case of space-only white noise hasn't received the same attention. In \cite{uemura1996construction} the author has showed that in this one-dimensional setting, equation (\ref{SHE}) admits a unique weak-solution which is square integrable for any $(t,x)\in[0,T]\times\mathbb R$; such a solution is constructed employing the Wiener Chaos expansion (see also Theorem $3.9$ of \cite{hu2015stochastic} ).

In \cite{hu2001heat}\cite{hu2002chaos} the author treated the SHE with time-homogeneous noise in the $d$-dimensional case, showing existence, and providing numerous estimations of the Lyapunov exponentials of the solutions. The former treats the case of fractional Brownian motion, while in the latter the solution is showed to exist in a \textit{flat} Hilbert space similar to those introduced by Kondratiev (see for instance \cite{holden2009stochastic}).

In \cite{hu2015stochastic}  the authors studied, among other things, the existence and regularity of the multidimensional version of (\ref{SHE}) when the covariance structure of $\{W_x\}_{x\in\mathbb R^d}$ satisfies certain conditions. They also propose some formal Feynman-Kac representations for the solutions of the SHE with space-time, and time-homogeneous Gaussian noise both in the Skorohod and Stratonovich sense. Nevertheless no representation was proposed for the solution of (\ref{SHE}).

In \cite{kim2017time} and \cite{kim2019stochastic} the authors study the SHE with time-homogeneous white noise potential in a bounded interval using the concept of Wiener chaos solution and propagator introduced in \cite{mikulevicius1993separation} (see also \cite{lototsky1997nonlinear}). They obtain estimation of the regularity of the solutions and well as existence and uniqueness results.

 The aim of this article is to construct a Feynman-Kac representation for the unique solution of (\ref{SHE}), this hasn't been achieved in  \cite{hu2015stochastic} due to the great generality under which the authors analyze the problem.  In this simpler framework we were able to use some of the techniques proposed by them to construct such a representation. 
 
 Our analysis relies on the use of a ``Wong-Zakai-type" approximation of equation (\ref{SHE}) where we replace the singular white noise $\{\dot W_x\}$ with a truncated Karuhnen-Loève-like series. 
 
 As in  \cite{lanconelli2021wong} this ``approximating equation'' is reduced to a deterministic partial differential equation for which we are able to derive a Feynman-Kac representation. Then we show that as the number of terms in the truncated series goes to infinity this representation converges to a well-defined random variable (for fixed $t$ and $x$) and that this limit-object is the unique solution of (\ref{SHE}) present in the literature. It's worth noticing that due to the structure of the approximated noise the sequence of ``approximated solutions'' converges not only in the $\mathbb L^p$ norm for any $p\in[1,\infty)$, but also almost surely.
 
\section{Preliminaries}
	
Let $(\Omega,\mathfrak B,\mathbb P^W)$ be a complete probability space which carries a one dimensional Brownian motion $\{W_x\}_{x\in\mathbb R}$ indexed by the real line.
Then consider the following Gaussian Hilbert space
\begin{align*}
\mathfrak H(W):=\bigg\{\int_{\mathbb R}h(x) d W_x; h\in \mathbb \mathbb L^2(\mathbb R)\bigg\},
\end{align*}
where the stochastic integral over the real line is defined in \cite{janson1997gaussian} (Chapter 7, section 2).
From now on we will let $\mathfrak B=\sigma(\mathfrak H(W))$, i.e. the sigma algebra generated by the Gaussian Hilbert space $\mathfrak H(W)$ then we have that the family of ``stochastic exponentials'' (also known as Wick exponentials)
\begin{align*}
\bigg\{\mathcal E^f:=\exp\left(\int_{\mathbb R}f(x)dW_x-\frac 1 2 \int_{\mathbb R}|f(x)|^2dx\right), f\in\mathbb L^2(\mathbb R)\bigg\},
\end{align*}
is total in $\mathbb L^2(\Omega,\mathcal B,\mathbb P^W)$ ($\mathbb L^2(P)$ for short). According to the Wiener chaos decomposition, any random variable  $X\in \mathbb L^2(P)$ can be represented as:
\begin{align*}
X=\sum_{n=0}^{\infty} I_n(h_n),\; \text{convergence in } \mathbb L^2(\mathbb P^W)
\end{align*} 
where $I_n(\bullet)$ denotes the $n$-th multiple stochastic integral (e.g. \cite{janson1997gaussian} Theorem 7.26), and the kernels $h_n\in \mathbb \mathbb L^2(\mathbb R^n)$ are symmetric deterministic functions.

If $A:\mathbb L^2(\mathbb P^W)\to \mathbb L^2(\mathbb P^W)$ is a bounded linear operator and we assume that $A$ is a contraction , then its ``second quantization operator'' $\Gamma(A):\mathbb L^2(\mathbb P^W)\to \mathbb L^2(\mathbb P^W)$ is given by the action:
\begin{align*}
\Gamma(A)X=\sum_{n=0}^{\infty} I_n(A^{\otimes n}h_n),
\end{align*}
notice that $A$ being a contraction is a sufficient condition for $\Gamma(A)$ to map $\mathbb L^2(\mathbb P^W)$ into itself.

In the following we will also need an complete orthonormal system (CONS for short) of the Hilbert space $\mathbb L^2(\mathbb R)$. In particular we will use the family of \textit{Hermite functions}, $\{e_j\}_{j\in\mathbb N}$ that are defined by:
\begin{align}
e_j(x)=(-1)^{j-1}\left(\sqrt{\pi}2^{j-1} (j-1)!\right)^{-1/2}e^{-x^2/2}\left(\frac{d}{dx}\right)^{j-1}e^{-x^2}, j\in\mathbb N
\end{align}
where we used a different indexing in order to avoid having a $0$-order element.

It's straightforward to see that taking tensor products we obtain a CONS for the space $\mathbb \mathbb L^2(\mathbb R^n)$, i.e.  $\{e_{i_1}\otimes\cdots\otimes e_{i_n}\}_{(i_1,...,i_n)\in\mathbb N^n}$.

 One could show that the family $\{e_j\}_{j\in\mathbb N}$ belongs to the Schwartz space of rapidly decreasing functions $S(\mathbb R)$, and this is particularly useful since this implies that we can grab any element of its dual space $S'(\mathbb R)$ (space of tempered distributions) and expand it into series of Hermite functions (e.g. \cite{picard1991hilbert}). In particular we will work with the Dirac's delta function $\delta_x\in S'(\mathbb R), x\in\mathbb R$, that can be written as:
\begin{align}\label{diracdelta}
\delta_x(\bullet)=\sum_{n=0}^{\infty} e_j(x)e_j(\bullet),
\end{align}
where the series clearly diverges in the classical sense. 

We will now briefly discuss a particular type of product between random variables; let $X, Y\in \mathbb L^2(\mathbb P^W)$ with Wiener chaos expansions given by 
\begin{align*}
		X=\sum_{n=0}^{\infty} I_n(h_n),\;\;\; 		Y=\sum_{n=0}^{\infty} I_n(g_n)
	\end{align*} 
the new element
\begin{align*}
X\diamond Y=\sum_{n=0}^{\infty} I_n(k_n),\;\; k_n:=\sum_{j=1}^{n} h_j\odot g_{n-j},
\end{align*}
where $``\odot "$ denotes the symmetric tensor product is called the Wick product of $X$ and $Y$. It's  worth noticing that $\mathbb L^2(\mathbb P^W)$ is not closed under Wick multiplication. This product has great relevance in several areas of stochastic analysis and quantum physics, the interested reader is referred to \cite{holden2009stochastic} and \cite{kuo2018white} from a complete account on the Wick product.

Finally all throughout this article we will denote with $p_{t-s}(x-y)$ the heat kernel
\begin{align*}
p_{t-s}(x-y):=\frac{1}{\sqrt{2\pi (t-s)}}e^{-\frac{(x-y)^2}{2(t-s)}},
\end{align*}
and $(P_tf)(x)$ will denote the action of the heat semigroup on the function $f$, i.e.
\begin{align*}
(P_t f)(x)=\int_{\mathbb R} p_{t}(x-y)f(y)dy.
\end{align*}

\section{Construction of the approximating equation.}

From now on we will work on the probability space $(\Omega,\mathfrak B,\mathbb P^W)$ introduced in the previous section. In this section we will propose an approximation of equation (\ref{SHE}),  and thus the first thing to do is to construct an opportune smooth approximation of the singular White noise process $\{\dot W_x\}_{x\in \mathbb R}$. It's known that the singular white noise at $x\in\mathbb R$ could be formally seen as 
	\begin{align*}
\dot W_x=\int_{\mathbb R} \delta_x(y)dW_y,
	\end{align*}
i.e. the stochastic integral of a Dirac delta function with mass at $x\in\mathbb R$ (see \cite{kuo2018white}).
	
One possible approximation can be obtained by truncating the series in (\ref{diracdelta}) up to a certain finite value $K$ yielding:
	\begin{align}\label{smooth WN}
		\dot W_x^K:=\sum_{j=1}^{K}e_j(x)\int_{\mathbb R}e_j(y)dW_y,
	\end{align}
(notice that the latter is nothing more than the derivative of a truncated Karuhnen-Loève expansion of the Brownian motion $W$)
clearly $\dot W^K$ converges to $\dot W$ in some space of generalized random variables (see formula $2.3.33$ of \cite{holden2009stochastic}).

If we substitute the singular white noise in (\ref{SHE}) with (\ref{smooth WN}) we obtain the following ``approximating equation'':
	\begin{align}\label{approx SHE}
		\begin{cases}
			\partial_t u_{t,x}^K=\frac 1 2 \partial_{xx}^2 u_{t,x}^K+u_{t,x}^K\diamond \dot W_x^K,\; (t,x)\in \mathbb [0,T]\times\mathbb R\\
			u(0,x)=u_0(x).
		\end{cases}
	\end{align}
	
Since the equation above involves non-trivial operations, such as taking the Wick product between the solution and a random potential we should state what a ``solution'' of the latter actually is.
Following  \cite{hu2011feynman} we have:
\begin{definition}\label{definition weak solution}
		Let $K$ be any arbitrary positive integer, then random field $u^K:[0,T]\times\mathbb R\times\Omega\to\mathbb R$ is said to be a weak solution of (\ref{approx SHE}) if for any fixed $(t,x)\in[0,T]\times\mathbb R$ we have that  $u^K_{t,x}\in \mathbb L^2(\mathbb P^W)$ and for any random variable $F\in\mathbb D^{1,2}$ it holds that:
		\begin{align}\label{definition}
			\mathbb E\left[u^K_{t,x} F\right]&= (P_t u_0)(x)\mathbb E[F]\nonumber\\
			&+\mathbb E\left[\bigg\langle \int_0^t\int_{\mathbb R} p_{t-s}(x-y) \left(\sum_{j=1}^Ke_j(y)e_j(\bullet)\right)u_{s,y}^Kdy\; ds,D_{(\bullet)}F\bigg\rangle_{\mathbb \mathbb L^2(\mathbb R)} \right],
		\end{align}
	where $D$ denotes the Malliavin derivative and $\mathbb D^{1,2}$ the Sobolev-Malliavin space \cite{nualart2006malliavin}.
\end{definition}

If in (\ref{approx SHE}) we consider a space-independent potential $\dot W_t^K$  the change of variables $v^K(t, x) = u^K(t, x) \diamond \exp^{\diamond}\left(\int_0^t W_s^Kds\right)$ (see section $3.6$ of \cite{holden2009stochastic} for a detailed explanation) reduces (\ref{approx SHE}) to a standard homogeneous heat equation.
Unfortunately, since the potential in our case is time-homogeneous (space-dependent) this approach is not applicable.
For this reason in order to deal with equation (\ref{approx SHE}) we will employ the approach proposed  \cite{lanconelli2021wong} (see also \cite{lanconelli2021small} ). The following remark offers a brief explanation of the latter.

\begin{remark}
	From now on we will let $Z_j(\omega):=\int_{\mathbb R}e_j(y)dW_y(\omega), j\in\{1,...,K\},\; \omega\in\Omega$, then we can rewrite equation (\ref{approx SHE}) as
	\begin{align}
	\begin{cases}
		\partial_t u_{t,x}^K=\frac 1 2 \partial_{xx}^2 u_{t,x}^K+\sum_{j=1}^K  e_j(x) u_{t,x}^K\diamond Z_j,\; (t,x)\in \mathbb [0,T]\times\mathbb R\\
		u(0,x)=u_0(x).
	\end{cases}
\end{align}

	It's known (e.g. formula $(2.5)$ of \cite{hu1996wick} or Theorem $9.20$  of \cite{kuo2018white} for an alternative but equivalent formulation) that the Wick product between a random variable $X\in\mathbb D^{1,2}$  and a random variable $I(g),\, g\in\mathbb L^2(\mathbb R)$ in first Wiener Chaos is given by:
	\begin{align*}
		X\diamond I(g)=X\cdot I(g)-\langle DX,g\rangle_{\mathbb \mathbb L^2(\mathbb R)}.
	\end{align*}

	Then proceeding formally we can write:
	\begin{align*}
		u_{t,x}^K\diamond Z_j= u_{t,x}^K \cdot Z_j-\langle Du_{t,x}^K,e_j\rangle_{\mathbb \mathbb L^2(\mathbb R)},
	\end{align*}
	and thus we are left to consider the equation
	\begin{align*}
		\begin{cases}
			\partial_t u_{t,x}^K=\frac 1 2 \partial_{xx}^2 u_{t,x}^K+u_{t,x}^K \cdot Z_j-\langle Du_{t,x}^K,e_j\rangle_{\mathbb \mathbb L^2(\mathbb R)},\; (t,x)\in \mathbb [0,T]\times\mathbb R\\
			u(0,x)=u_0(x).
		\end{cases}
	\end{align*}
	
	If we assume that the solution of the equation above is of the form 
	\begin{align*}
		u_{t,x}^K(\omega)=\mathfrak u^K(t,x,Z_1(\omega),...,Z_K(\omega))
	\end{align*}
	for some continuous function $\mathfrak u:[0,T]\times\mathbb R\times\mathbb R^K\to\mathbb R$ that must be determined
	we can use the chain rule for Malliavin derivatives and obtain the following partial differential equation (PDE for short):
	\begin{align*}
		\begin{cases}
			\partial_t \mathfrak u_{t,x,z}^K=\frac 1 2 \partial_{xx}^2 \mathfrak u_{t,x,z}^K+\mathfrak u_{t,x,z}^K\left(\sum_{j=1}^{K}e_j(x)z_j\right)-\sum_{j=1}^{K}e_j(x)\partial_{z_j}\mathfrak u_{t,x,z}^K,\; (t,x,z)\in [0,T]\times\mathbb R\times\mathbb R^K\\
			\mathfrak u^K(0,x)=u_0(x).
		\end{cases}
	\end{align*}
Upon multiplying both sides of the equation by $\exp\left(-\frac{1}{2}\sum_{j=1}^{K}z_j^2\right)$ and defining $\mathfrak v_{t,x,z}^K:=\mathfrak u_{t,x,z}^K  \exp\left(-\frac{1}{2}\sum_{j=1}^{K}z_j^2\right)$, we are able to get rid of the zero-order term above and obtain the following:
	\begin{align}\label{PDE2}
		\begin{cases}
			\partial_t \mathfrak v^K=\frac{1}{2}\partial_{xx}^2 \mathfrak v^K-\sum_{j=1}^{K}e_j(x)\partial_{z_j}\mathfrak v^K,\\
			\mathfrak v^K(0,x)=u_0(x)\times \exp\left(-\frac{1}{2}\sum_{j=1}^{K}z_j^2\right).
		\end{cases}
	\end{align}
	
Formally applying the classical Feynman-Kac formula we obtain:
	\begin{align}\label{FeynmanKac PDE}
		\mathfrak v^{K}(t,x,z)=\mathbb{E}^B\left[u_0(B_t^x) \exp\bigg\{-\frac{1}{2}\sum_{j=1}^{K}\left(z_j
		-\int_0^t e_j(B_s^x)ds\right)^2\bigg\}\right],
	\end{align}
where $\{B_t\}_{t\in [0,T]}$ is a $1$-dimensional Brownian motion defined on the auxiliary filtered probability space $(\tilde{\Omega},\mathcal G,\{\mathcal G^B_t\}_{t\in[0,T]},\mathbb P^B)$ and $\mathbb E^B$ denotes the expectation in this space.

The expression given by (\ref{FeynmanKac PDE}) is sometimes referred to  as a ``generalized solution'' of (\ref{PDE2}). It is worth mentioning that the latter becomes a classical solution if suitable regularity assumptions on the coefficients of
(\ref{PDE2}) are in force. For more details see \cite{freidlin2016functional} (page 122).
By definition the latter implies that
	\begin{align*}
		\mathfrak u^{K}(t,x,z)=\mathbb{E}^B\left[u_0(B_t^x) \exp\bigg\{\sum_{j=1}^{K}z_j\left(\int_0^te_j(B_s^x)ds\right)-\frac{1}{2}\sum_{j=1}^{K}\left(\int_0^te_j(B_s^x)ds\right)^2\bigg\}\right].
	\end{align*}
	
	Letting $u^K_{t,x}(\omega):=\mathfrak u^K(t,x,Z_1(\omega),...,Z_K(\omega))$ we obtain the formula in theorem \ref{existence}.
\end{remark}

\section{Statements of the theorems.}
\begin{theorem}\label{existence}
		The random field $u^K:[0,T]\times \mathbb R\times \Omega\to\mathbb R$ defined by:
		\begin{align}\label{approx FK}
			u_{t,x}^K=\mathbb{E}^B\left[u_0(B_t^x)\times \exp\bigg\{\int_{\mathbb R}\sum_{j=1}^{K}\left(\int_0^te_j(B_s^x)ds\right) e_j(y)dW(y)- \frac{1}{2}\sum_{j=1}^{K}\left(\int_0^te_j(B_s^x)ds\right)^2\bigg\}\right],
		\end{align}
		is a weak solution of (\ref{approx SHE}) (in the sense of definition \ref{definition weak solution} )
\end{theorem}

\begin{theorem}\label{Feynman-Kac functional}
		The family of random variables $\{\Psi_{t,x}^K;\; K\in\mathbb N\}$ given by:
		\begin{align*}
			\Psi_{t,x}^K:=\int_{\mathbb R}\sum_{j=1}^{K}\left(\int_0^te_j(B_s^x)ds\right) e_j(y)dW(y)-\frac{1}{2}\sum_{j=1}^{K}\left(\int_0^te_j(B_s^x)ds\right)^2,
		\end{align*}
		converges in $\mathbb L^2(\mathbb P^W\otimes \mathbb P^B)$ to a well defined random variable denoted by 
		\begin{align*}
			\Psi_{t,x}=\int_0^t\int_{\mathbb R} \delta(B_s^x-y)\; d W(y)\;ds-\frac{1}{2}\int_{\mathbb R}L_a(t)^2da.
		\end{align*}
	Furthermore conditional on $\mathcal G_T^B$ it holds that 
	$\Psi_{t,x}\sim N\left(-\frac{1}{2}\int_{\mathbb R}|L_a(t)|^2da,\int_{\mathbb R}|L_a(t)|^2da\right)$,
	where $\{L_a(t); (t,a)\in[0,T]\times\mathbb R\}$ is the local time of $\{B_t\}_{t\in [0,T]}$.
\end{theorem}

\begin{theorem}\label{Feynman-Kac}
For fixed $(x,t)\in [0,T]\times\mathbb R$,  $p\in[1,\infty)$, let $u_{t,x}^K:=\mathbb E^B\left[u_0(B_t^x)\exp\{\Psi_{t,x}^K\}\right]$ and denote $u_{t,x}:=\mathbb E^B\left[u_0(B_t^x)\exp\{\Psi_{t,x}\}\right]$, then it holds that:
\begin{align}
\lim_{K\to\infty} \|u_{t,x}^K-u_{t,x}\|_{\mathbb L^p(\mathbb P^W)}=0,
\end{align}
and 
\begin{align}
	\lim_{K\to\infty} u_{t,x}^K=u_{t,x},\; \mathbb P^W-a.s.
\end{align}

 Furthermore we have that 
\begin{align*}
[0,T]\times \mathbb R\times \Omega\ni (t,x,\omega)\mapsto u_{t,x}(\omega)
\end{align*}
 is the unique solution for (\ref{SHE}) (in the sense of theorem $3.1$ of \cite{uemura1996construction}).
\end{theorem}

Using our Feynman-Kac representation  we are able to derive the following formulae for the moments of the solution (this formula has also been obtained in \cite{hu2015stochastic} but no proof is provided).
\begin{theorem}\label{Moments}
Let $q\geq 2$ then the $q$-th moment of the unique solution of (\ref{SHE}) is given by
\begin{align}
\mathbb{E}^W\left[\left(u_{t,x}\right)^q\right]=\mathbb{E}^B\left[\left(\prod_{i=1}^qu_0(B_t^{(i)}+x)\right)\exp\bigg\{\sum_{i<j}^{q}\int_0^t\int_0^t \delta_0(B_s^{(i)}-B_r^{(j)})dsdr\bigg\}\right],
\end{align}
$(B^{(1)},...,B^{(q)})$ are $q$ independent $1$-dimensional Brownian motions  and $\int_0^t\int_0^t\delta_0(B_s-B_r')dsdr$  denotes the  intersection local time of the Brownian motions $B$ and $B'$(e.g. \cite{le1994exponential})
\end{theorem}

\section{Proof of theorem \ref{existence}}

	Letting $F=\mathcal E^{\xi},\;\xi\in\mathbb L^2(\mathbb R)$ we have that $\mathbb E\left[u_{t,x}^K F\right]$ coincides with the $S$-transform (e.g.\cite{kuo2018white}) of $u_{t,x}^K$; i.e.
	\begin{align*}
		\mathbb E^W\left[u_{t,x}^K F\right]=\mathcal S (u_{t,x}^K)(\xi)&=\mathbb E^W\left[\mathbb{E}^B\left[u_0(B_t^x)\times \exp\{\Psi_{t,x}^K\}\right]\times\mathcal E^{\xi}\right]\\
		&=\mathbb E^B\left[u_0(B_t^x) \mathbb{E}^W\left[\exp\{\Psi_{t,x}^K\}\times\mathcal E^{\xi}\right]\right]\\
		&=\mathbb E^B\left[u_0(B_t^x) \exp\bigg\{\int_0^t \sum_{j=1}^{K}e_j(B_s^x)\int_{\mathbb R}e_j(y)\xi(y)dy\;ds \bigg\}\right],
	\end{align*}
	hence by the classical Feynman-Kac formula we can see that the latter is the solution of 
	\begin{align*}
		\begin{cases}
			\partial_t \mathcal S_{t,x}(\xi)=\frac{1}{2} \partial_{xx}^2 \mathcal S_{t,x}(\xi)+ \mathcal S_{t,x}(\xi)\cdot\left( \sum_{j=1}^{K}e_j(x)\int_{\mathbb R}e_j(y)\xi(y)dy\right), \; (t,x)\in [0,T]\times\mathbb R\\
			\mathcal S_{0,x}(\xi)=u_0(x),
		\end{cases}
	\end{align*}
	for any $\xi\in \mathbb L^2(\mathbb R)$.
	
	The solution of this equation can be written in mild form as	
	\begin{align*}
		\mathcal S_{t,x}(\xi)&=(P_t u_0)(x)+\int_0^t \int_{\mathbb R}p_{t-s}(x-y)\mathcal S_{s,y}(\xi) \sum_{j=1}^{K}e_j(y)\int_{\mathbb R}e_j(z)\xi(z)dz\; dy,\\
		&=(P_t u_0)(x)+\int_0^t \int_{\mathbb R}p_{t-s}(x-y)\mathbb E^W\left[u_{t,x}^K \mathcal E^{\xi}\right]\sum_{j=1}^{K}e_j(y)\int_{\mathbb R}e_j(z)\xi(z)dz\; dy,
	\end{align*}	
	or which is equivalent,	
	\begin{align*}
		\mathbb E^W\left[u_{t,x}^K \mathcal E^{\xi}\right]=(P_t u_0)(x)+ \mathbb E^W\left[\bigg\langle \int_0^t \int_{\mathbb R}p_{t-s}(x-y) \left(\sum_{j=1}^{K}e_j(z)e_j(\bullet )\right) u_{s,y}^K dyds,\xi(\bullet )\mathcal E^{\xi} \bigg\rangle_{\mathbb \mathbb L^2(\mathbb R)}\right],
	\end{align*}
	which implies (\ref{definition}) since $\mathbb E^W[\mathcal E^{\xi}]=1$ and $D_{(\bullet)} \mathcal E^{\xi}=\xi(\bullet)\mathcal E^{\xi}$ and the fact that the stochastic exponentials are a dense family in $\mathbb D^{1,2}$.

\section{Proof of theorem \ref{Feynman-Kac functional}}

	We start by showing that  $\{\Psi_{t,x}^K\}_{K\in\mathbb N}$ is a Cauchy sequence in $\mathbb L^2(\mathbb P^W\otimes \mathbb P^B)$, and we let $\vvvert\bullet\vvvert_p$ be the norm on the Hilbert space $\mathbb L^p(\mathbb P^W\otimes \mathbb P^B)$.
	
	Without lost of generality assume that $N\geq M$
	\begin{align*}
		\vvvert\Psi_{t,x}^N-\Psi_{t,x}^M\vvvert_2^2&=\left\vvvert\sum_{j=M+1}^{N}z_j\times\left(\int_0^te_j(B_s^x)ds\right)-\frac{1}{2}\sum_{j=1}^{K}\left(\int_0^te_j(B_s^x)ds\right)^2\right\vvvert_2 ^2\\
		&= \mathbb{E}^B\mathbb{E}^W\left[\bigg|\sum_{j=M+1}^{N}z_j\times\left(\int_0^te_j(B_s^x)ds\right)-\frac{1}{2}\sum_{j=M+1}^{N}\left(\int_0^te_j(B_s^x)ds\right)^2\bigg|^2\right]\\
		&=\mathbb{E}^B\mathbb{E}^W\Bigg[\left(\sum_{j=M+1}^{N}z_j\times\left(\int_0^te_j(B_s^x)ds\right)\right)^2\\
		&-\left(\sum_{j=M+1}^{N}z_j\times\left(\int_0^te_j(B_s^x)ds\right)\right)\left(\sum_{j=M+1}^{N}\left(\int_0^te_j(B_s^x)ds\right)^2\right)\\
		&+\frac 1 4 \left(\sum_{j=M+1}^{N}\left(\int_0^te_j(B_s^x)ds\right)^2\right)^2\Bigg]\\
		&=\mathbb{E}^B\left[\sum_{j=M+1}^{N}\left(\int_0^te_j(B_s^x)ds\right)^2+\frac 1 4 \left(\sum_{j=M+1}^{N}\left(\int_0^te_j(B_s^x)ds\right)^2\right)^2\right].
	\end{align*}
	
Then if would suffices to show that 	
	\begin{align*}
		\lim_{N,M\to\infty} \mathbb{E}^B\left[\sum_{j=M+1}^{N}\left(\int_0^te_j(B_s^x)ds\right)^2+\frac 1 4 \left(\sum_{j=M+1}^{N}\left(\int_0^te_j(B_s^x)ds\right)^2\right)^2\right]=0.
	\end{align*}
	
Using (\ref{local time parseval}) we have that 
	\begin{align*}
		\sum_{j=M+1}^{N}\left(\int_0^te_j(B_s^x)ds\right)^2\leq \int_{\mathbb R} |L_a(t)|^2 da=\alpha_t,
	\end{align*}
for any positive integers $N\geq M$. This together with the fact that the random variable $\alpha_t$ is exponentially integrable (e.g. \cite{le1994exponential} p. $178$) allows us to use the Dominated Convergence theorem to bring the limit inside the expectation. 

Finally from (\ref{local time parseval}) we know that the sequence 
	\begin{align*}
		S_n=\sum_{j=1}^{n}\left(\int_0^te_j(B_s^x)ds\right)^2,
	\end{align*}
is $\mathbb P^B$-a.s. convergent and thus we have that 
	\begin{align*}
		|S_N-S_M|=\sum_{j=M+1}^{N}\left(\int_0^te_j(B_s^x)ds\right)^2\to 0,\; \mathbb P^B-a.s.
	\end{align*}
	when $N,M\to\infty$.
	
	This implies that 
	\begin{align*}
		\lim_{N,M\to\infty}\vvvert\Psi_{t,x}^N-\Psi_{t,x}^M\vvvert_2^2=0.
	\end{align*}

	Furthermore notice that conditional on $\mathcal G_T^{B}$ the random variable $\Psi_{t,x}^N$ is a Gaussian random variable with mean $-\frac{1}{2}\sum_{j=1}^{K}\left(\int_0^te_j(B_s^x)ds\right)^2$ and variance $\sum_{j=1}^{K}\left(\int_0^te_j(B_s^x)ds\right)^2$ and since the $\mathbb L^2(P^W)$ limit preserves the Gaussianity we have that conditional on $\mathcal G_T^{B}$ the random variable 
	$\Psi_{t,x}\sim N\left(-\frac{1}{2}\int_{\mathbb R}|L_a(t)|^2da,\int_{\mathbb R}|L_a(t)|^2da\right)$.
	
\section{Proof of theorem \ref{Feynman-Kac}}

The proof of our main theorem will be done in several steps:
\textit{
\begin{enumerate}
\item Show that the ``approximated Feynman-Kac'' formula converges in $\mathbb L^2(\mathbb P^W)$ to the ``formal Feynman-Kac''.
\item Obtain the Wiener Chaos expansion of the ``approximated Feynman-Kac''.
\item Show that the latter converges in $\mathbb L^2(\mathbb P^W)$ (as $K\to\infty$) to the solution of (\ref{SHE}) represented by the formal series given in \cite{uemura1996construction} and \cite{hu2002chaos}.
\end{enumerate}}

Then since the limit in $\mathbb L^2(\mathbb P^W)$ is $\mathbb P^W$-a.s. unique, we conclude that the solution given by the formal Wiener chaos series in \cite{uemura1996construction} and \cite{hu2002chaos} coincides with the ``formal Feynman-Kac'' formula.

The previous can be summarized by the following diagram,
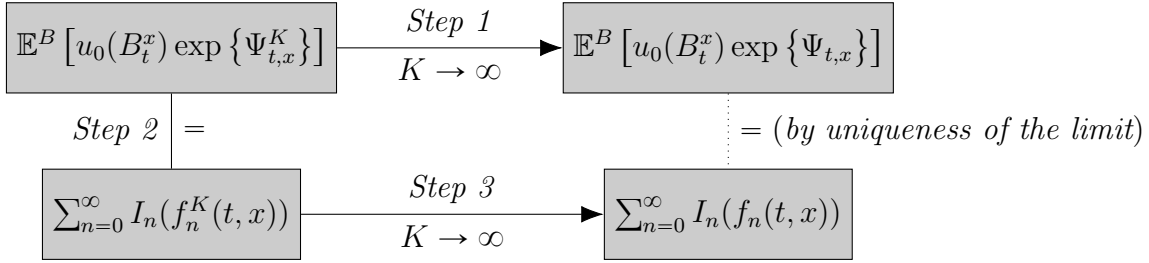
\begin{center}
	\begin{tikzpicture}
		\definecolor{blue}{gray}{0.8}
		
		\node [draw, fill=blue,
		minimum width=2cm,
		minimum height=1.2cm,
		]  (FK1) {$\mathbb{E}^B\left[u_0(B_t^x)\exp\big\{\Psi_{t,x}^K\big\}\right]$};
		
		\node [draw,fill=blue,
		minimum width=2cm, 
		minimum height=1.2cm,
		right=3cm of FK1
		] (FK2) {$\mathbb{E}^B\left[u_0(B_t^x)\exp\big\{\Psi_{t,x}\big\}\right]$};
		
		\node [draw,fill=blue,
		minimum width=2cm, 
		minimum height=1.2cm, 
		below = 1cm of FK1
		]  (WC1) {$\sum_{n=0}^{\infty}I_n(f_n^K(t,x))$};

		\node [draw,fill=blue,
		minimum width=2cm, 
		minimum height=1.2cm, 
		below = 1cm of FK2
		]  (WC2) {$\sum_{n=0}^{\infty}I_n(f_n(t,x))$};


		\draw[-{Latex[length=3mm]}] (FK1) -- (FK2)
		node[midway,above]{\textit{Step 1}}
		node[midway,below]{$K\to\infty$}
		;
		
		\draw[-{Latex[length=3mm]}] (WC1) -- (WC2)
		node[midway,above]{\textit{Step 3}}
		node[midway,below]{$K\to\infty$};
		
		\draw[-] (FK1) -- (WC1)
		node[midway,left]{\textit{Step 2}}
		node[midway,right]{$=$};
		
		\draw[dotted] (FK2) -- (WC2)
		node[midway,right]{$\mathbf{=}\;(\textit{by uniqueness of the limit})$};
	\end{tikzpicture}
\end{center}

\subsection*{\textit{Step 1}:}
We start by showing the convergence of the ``approximated Feynman-Kac'' formula:
\begin{align*}
	\mathbb E^W\left[ |u_{t,x}^K-u_{t,x}|^p\right]&=\mathbb E^W \big|\mathbb E^B \left[u_0(B_t^x)\left(\exp\{\Psi_{t,x}^K\}-\exp\{\Psi_{t,x}\}\right)\right]\big|^p\\
	&\leq \|u_0\|_{\infty}^p\mathbb E^W \mathbb E^B\left[\big|\exp\{\Psi_{t,x}^K\}-\exp\{\Psi_{t,x}\} \big|^p\right]
\end{align*}
Since $\Psi_{t,x}^K\to \Psi_{t,x}$ in $\mathbb L^2(\mathbb P^W\otimes P^B)$, then 
$\exp\{\Psi_{t,x}^K\}\to\exp\{\Psi_{t,x}\}$ in probability, in order to show the desired result we just need to prove that $\vvvert \exp\{\Psi_{t,x}^K\}\vvvert_p\to \vvvert \exp\{\Psi_{t,x}\}\vvvert_p$.

Using the Tower rule and the fact that conditional on $\mathcal G_T^B$ the random variables $\Psi_{t,x}$ and $\Psi_{t,x}^K$ are Gaussian we have that 
\begin{align*}
	\vvvert\exp\{\Psi_{t,x}\}\vvvert^p_p &=\mathbb E^W\mathbb E^B|\exp\{p\Psi_{t,x}\}|\\
	&=\mathbb E^B\left[\mathbb E^W\left[\exp\{p\Psi_{t,x}\}|\mathcal G_T^B\right]\right]\\
	&=\mathbb E^B\left[\exp\bigg\{\frac{p(p-1)}{2}\int_{\mathbb R}|L_a(t)|^2da\bigg\}\right]<\infty,
\end{align*}
and
\begin{align*}
	\vvvert\exp\{\Psi_{t,x}^K\}\vvvert^p_p &=\mathbb E^W\mathbb E^B|\exp\{p\Psi_{t,x}^K\}|\\
	&=\mathbb E^B\left[\mathbb E^W\left[\exp\{p\Psi_{t,x}\}|\mathcal F_T^B\right]\right]\\
	&=\mathbb E^B\left[\exp\bigg\{\frac{p(p-1)}{2}\sum_{j=1}^{K}\left(\int_0^te_j(B_s^x)ds\right)^2\bigg\}\right].
\end{align*}
At this point we can use Monotone convergence theorem to bring the limit inside the expectation, the continuity of the exponential function and (\ref{local time parseval}) implies the desired result.

\subsection*{\textit{Step 2}:}

Now we need to obtain the Wiener chaos decomposition of (\ref{approx FK}) and start by noticing that conditional on $\mathcal G_T^B$ we can write
\begin{align*}
	\exp\{\Psi_{t,x}^K\}=\sum_{n=0}^{\infty}\frac{1}{n!}I_n\left(g_n^K(t,x,\bullet)\right),\;\text{convergence in } \mathbb L^2(\mathbb P^W)
\end{align*}
where the $n$-th kernel is given by:
\begin{align*}
	g^K_n(t,x,\bullet)=\sum_{i_1=1}^{K}\cdots\sum_{i_n=1}^{K}\frac{1}{n!}\left(\int_{[0,t]^n}e_{i_1}(B_{s_1}^x)\cdots e_{i_n}(B_{s_n}^x)ds_1...ds_n\right) (e_{i_1}\otimes\cdots\otimes e_{i_n})(\bullet).
\end{align*}
From the latter it follows that 
\begin{align*}
	u_{t,x}^K&=\mathbb{E}^B\Bigg[u_0(B_t^x) \times \sum_{n=0}^{\infty}I_n\left(\sum_{i_1=1}^{K}\cdots\sum_{i_n=1}^{K}\frac{1}{n!}\left(\int_{[0,t]^n}e_{i_1}(B_{s_1}^x)\cdots e_{i_n}(B_{s_n}^x)ds_1...ds_n\right) e_{i_1}\otimes\cdots\otimes e_{i_n}\right)\Bigg].
\end{align*}

Since the series is convergent in $\mathbb L^2(\mathbb P^W)$ we can apply Jensen inequality and monotone convergence to interchange the series with the expectation yielding
\begin{align*}
	\sum_{n=0}^{\infty}I_n\left(\sum_{i_1=1}^{K}\cdots\sum_{i_n=1}^{K}\mathbb E^B\left[u_0(B_t^x)\frac{1}{n!}\int_{[0,t]^n.}e_{i_1}(B_{s_1}^x)\times\cdots\times e_{i_n}(B_{s_n}^x)ds_1...ds_n\right] e_{i_1}\otimes\cdots\otimes e_{i_n}\right).
\end{align*}

Since we are summing over all possible combinations of the indexes $(i_1,...,i_n)$ the expression above equals
\begin{align*}
	\sum_{n=0}^{\infty}I_n\left(\sum_{i_1=1}^{K}\cdots\sum_{i_n=1}^{K}\mathbb E^B\left[\int_{\mathbb{T}_{t,n}}u_0(B_t^x)e_{i_1}(B_{s_1}^x)\times\cdots\times e_{i_n}(B_{s_n}^x)ds_1...ds_n\right] e_{i_1}\otimes\cdots\otimes e_{i_n}\right)
\end{align*}
where the time integrals are taken over the simplex
\begin{align*}
	\mathbb{T}_{t,n}:=\{(s_1,s_2,...,s_n);0\leq s_1\leq s_2\leq \cdots\leq s_n\leq t\}.
\end{align*}

An application of Fubini-Tonelli lemma shows that the kernel of the $n$-fold multiple stochastic integral is given by:
\begin{align*}
	\int_{\mathbf{T}_{t,n}}\sum_{i_1=1}^{K}\cdots\sum_{i_n=1}^{K}e_{i_1}\otimes\cdots\otimes e_{i_n}(\bullet)\mathbb E^B\left[u_0(B_t^x)e_{i_1}(B_{s_1}^x)\times\cdots\times e_{i_n}(B_{s_n}^x)\right] ds_1...ds_n,\\
\end{align*}
the conclusion is stated in the following proposition.
\begin{proposition}
	Let $u_{t,x}^K$ be given by (\ref{approx FK}) then it holds that:
	\begin{align}\label{WC1}
		u_{t,x}^K=\sum_{n=0}^{\infty} I_n(f_n^K(t,x)),
	\end{align}
	where 
	\begin{align}\label{kernel1}
		\begin{cases}
			f_0^K(t,x)&=(P_t u_0)(x),\\		f_n^K(t,x,\bullet)&=\int_{\mathbb{T}_{t,n}}\sum_{i_1=1}^{K}\cdots\sum_{i_n=1}^{K}e_{i_1}\otimes\cdots\otimes e_{i_n}(\bullet)\\
			&\times \int_{\mathbb R^{n+1}}p_{t-s_n}(x_n-x_{n-1})\cdots p_{s_1}(x_0-x) u_0(x_n)e_{i_1}(x_{n-1})\times\cdots\times e_{i_n}(x_0) d\mathbf{x}\; d\mathbf{s}
		\end{cases}
	\end{align}
	where $ d\mathbf{x}:=dx_0\cdots dx_n$, $d\mathbf{s}:=ds_1\cdots ds_n$.
\end{proposition}

In \cite{uemura1996construction} the author has shown that if the initial condition is deterministic and square integrable, then equation (\ref{SHE}) has a unique weak solution given by the Wiener Chaos expansion: 
\begin{align}\label{WC2}
	u(t,x)=\sum_{n=0}^{\infty} I_n(f_n(t,x)),
\end{align}
where
\begin{align}\label{kernels}
	\begin{cases}
		f_0(t,x)=(P_t u_0)(x),\\
		f_n(t,x;x_1,...,x_n)=\int_{\mathbb{T}_{t,n}}\int_{\mathbb R} p_{t-s_n}(x-x_n)\cdots p_{s_1}(x_1-x_0)u_0(x_0) dx_0 d\mathbf{s},
	\end{cases}
\end{align}
see also equations $(4.2)$ and $(4.3)$ of \cite{hu2002chaos}.

Now the idea is to show that (\ref{WC1}) converges in $\mathbb L^2(\mathbb P^W)$ to (\ref{WC2}) as $K\to \infty$ and for that we will need the following:

\begin{definition}\label{projection}
	Let $K$ be some fixed positive integer then $A_K:S'(\mathbb R)\to \mathbb \mathbb L^2(\mathbb R)$ is a self-adjoint projection operator defined by the action 
	\begin{align}
		A_K f=A_K \left(\sum_{j=1}^{\infty}\langle f,e_j\rangle e_j \right)=\sum_{j=1}^{K}\langle f,e_j\rangle e_j,
	\end{align}
	for any $f\in S'(\mathbb R)$, i.e.  the orthogonal projection on the linear span of the first $K$ elements of the CONS $\{e_j\}_{j\in\mathbb N}$.
\end{definition}

\begin{proposition}\label{second quantization}
	Let $u(t,x), (t,x)\in [0,T]\times \mathbb R$ denote the weak solution of (\ref{SHE}) given in \cite{uemura1996construction} (eq. $(3.5)$). Then for any $(t,x)\in[0,T]\times\mathbb R$ it holds that:
	\begin{align}
		u_{t,x}^K=\Gamma(A_K)u(t,x),
	\end{align}
	where $\Gamma(A_K)$ stands for the second quantization of the projection operator $A_K$.
\end{proposition}

\begin{proof}
	In order to prove the latter we need to show that 
	\begin{align*}
		f_n^K(t,x,\bullet)=(A_K^{\otimes n}f_n(t,x))(\bullet)=\sum_{i_1=1}^{K}\cdots\sum_{i_n=1}^{K}\bigg\langle f_n,e_{i_1}\otimes\cdots\otimes e_{i_n}\bigg\rangle_{\mathbb \mathbb L^2(\mathbb R^n)} e_{i_1}\otimes\cdots\otimes e_{i_n}(\bullet),
	\end{align*}
	where $f_n^K$ and $f_n$ are defined by (\ref{kernel1}) and (\ref{kernels}) respectively.
	
For the sake of simplicity we consider the case with $n=2$, the general case does not present particular difficulties besides the more complex notation.
In this case (\ref{kernel1}) takes the form:
	\begin{align*}
		f_2^K(t,x,\bullet)&=\int_0^t\int_0^{s_2} \sum_{i_1=1}^K\sum_{i_2=1}^Ke_{i_1}\otimes e_{i_2}(\bullet)\\
		&\times \int_{\mathbb R^3}p_{t-s_2}(x_2-x_1)p_{s_2-s_1}(x_1-x_0)p_{s_1}(x_0-x)u_0(x_2)e_{i_1}(x_1)e_{i_2}(x_0)d\mathbf{x}\;d\mathbf{s}.
	\end{align*}
	We will define a new set of variables according to the prescription:
	\begin{align*}
		r_1&:=t-s_2,\\
		r_2&:=t-s_1,\\
		y_0&:=x_2,\\
		y_1&:=x_1,\\
		y_2&:=x_0,
	\end{align*}
	then we can rewrite the expression above as
	\begin{align*}
		&\int_0^t\int_{r_1}^{t} \sum_{i_1=1}^K\sum_{i_2=1}^Ke_{i_1}\otimes e_{i_2}(\bullet) \int_{\mathbb R^3}p_{r_1}(y_1-y_0)p_{r_2-r_1}(y_2-y_1)p_{t-r_2}(x-y_2) u_0(y_0)e_{i_1}(y_1)e_{i_2}(y_2)d\mathbf{y}dr_2dr_1\\
		&=\int_0^t\int_{0}^{r_2} \sum_{i_1=1}^K\sum_{i_2=1}^Ke_{i_1}\otimes e_{i_2}(\bullet) \int_{\mathbb R^3}p_{t-r_2}(x-y_2)p_{r_2-r_1}(y_2-y_1)p_{r_1}(y_1-y_0)u_0(y_0)e_{i_1}(y_1)e_{i_2}(y_2)d\mathbf{y}\;d\mathbf r,
	\end{align*}
	which is equal to $(A_K^{\otimes 2}f_2(t,x))(\bullet)$.
From here it's easy to see that 
\begin{align*}
	u_{t,x}^K=\sum_{n=0}^{\infty} I_n(f_n^K(t,x))=\sum_{n=0}^{\infty} I_n(A_K^{\otimes n}f_n(t,x))=\Gamma(A_K)u(t,x),
\end{align*}
which proves the result.
\end{proof}

\subsection*{Step 3:}
It's straightforward to see that 
\begin{align*}
\left\|\sum_{n=0}^{\infty} I_n(f_n^K(t,x))-\sum_{n=0}^{\infty} I_n(f_n(t,x))\right\|_{\mathbb L^2(\mathbb P^W)}^2&=\left\|\sum_{n=0}^{\infty} I_n(f_n(t,x)-A^{\otimes n}_Kf_n(t,x))\right\|_{\mathbb L^2(\mathbb P^W)}^2\\
&=\sum_{n=0}^{\infty} n! \|f_n(t,x)-(A^{\otimes n}_Kf_n(t,x))\|_{\mathbb L^2(\mathbb R^n)}^{2}\to 0
\end{align*}
as $K\to\infty$. This together with the results obtained in \textit{Step 1} and the unicity of the $\mathbb L^2(\mathbb P^W)$ limit we conclude that
\begin{align*}
\mathbb{E}^B\left[u_0(B_t^x)\exp\big\{\Psi_{t,x}\big\}\right]=\sum_{n=0}^{\infty} I_n(f_n(t,x)).
\end{align*}

On the other hand from the propositions \ref{second quantization} and \ref{SQ and CE prop} we see that 
\begin{align*}
u_{t,x}^K=\mathbb{E}^W\left[u(t,x)|\sigma(Z_1,...,Z_K)\right],
\end{align*}
and we also notice that $\sigma(Z_1,...,Z_K)\uparrow \mathfrak B:=\sigma(\mathfrak{H}(W))$.
Then the martingale convergence theorem (e.g. theorem $35.6$ of \cite{billingsley2008probability}) gives us the $\mathbb P^W$-a.s. convergence

\section{Proof of theorem \ref{Moments}}

The convergence in $\mathbb L^p(\mathbb P^W),\; p\in [1,\infty)$ of $u_{t,x}^K$ to the solution of $u_{t,x}$ implies that for any $q\in \mathbb N$ the $q$-th moment of $u_{t,x}^K$ converges to that of the solution.
\begin{align*}
	\mathbb{E}^W\left[\left(u_{t,x}^K\right)^q\right]&=\mathbb{E}^W\left[\prod_{i=1}^q\mathbb E^B\left[u_0(B_t^{(i)}+x)\exp\big\{\Psi_{t,x}^{K,(i)}\big\}\right]\right]\\
	&=\mathbb{E}^B\left[\left(\prod_{i=1}^q u_0(B_t^{(i)}+x)\right)\mathbb E^W\left[\exp\bigg\{\sum_{i=1}^{q}\Psi_{t,x}^{K,(i)}\bigg\}\bigg|\mathcal F_T^B\right]\right]\\
	&=\mathbb{E}^B\bigg[\left(\prod_{i=1}^q u_0(B_t^{(i)}+x)\right)\mathbb E^W\left[\exp\bigg\{\int_{\mathbb R}\sum_{i=1}^{q}\sum_{k=1}^{K}\left(\int_0^te_k(B_s^{(i)}+x)ds\right)e_k(y) dW_y\bigg\}\bigg|\mathcal F_T^B\right]\\
	&\times \exp\bigg\{-\frac{1}{2}\sum_{i=1}^{q}\left(\int_0^t e_k(B_s^{(i)}+x)ds\right)^2\bigg\}\bigg],
\end{align*}

using the fact that conditional on $\mathcal G_T^B$ the stochastic integral appearing in the exponential is a centered Gaussian random variable we can see that the latter equals

\begin{align*}
	&=\mathbb{E}^B\bigg[\left(\prod_{i=1}^qu_0(B_t^{(i)}+x)\right) \exp\bigg\{\left\|\sum_{i=1}^{q}\sum_{k=1}^{K}\left(\int_0^te_k(B_s^{(i)}+x)ds\right)e_k(\bullet)\right\|_{\mathbb \mathbb L^2(\mathbb R)}^2\bigg\}\\
	&\times \exp\bigg\{-\frac{1}{2}\sum_{i=1}^{q}\left(\int_0^t e_k(B_s^{(i)}+x)ds\right)^2\bigg\}\bigg]\\
	&=\mathbb{E}^B\left[\left(\prod_{i=1}^qu_0(B_t^{(i)}+x)\right)\exp\bigg\{\sum_{i<j}^q\sum_{k=1}^{K}\left(\int_0^t e_k(B_s^{(i)}+x)ds\right)\left(\int_0^t e_k(B_r^{(j)}+x)dr\right)\bigg\}\right].
\end{align*}
Now we must take the limit for $K\to \infty$ (we can see that the exponential function is dominated by $\exp\{q\max_{1\leq i\leq q}\int_{\mathbb R}|L_a^{(i)}(t)|^2da\}$ which is integrable) yielding 

\begin{align}
	\mathbb{E}^W\left[\left(u_{t,x}\right)^q\right]=\lim_{K\to\infty} \mathbb{E}^W\left[\left(u_{t,x}^K\right)^q\right]=\mathbb{E}^B\left[\left(\prod_{i=1}^qu_0(B_t^{(i)}+x)\right)\exp\bigg\{\sum_{i<j}^{q}\int_0^t\int_0^t \delta_0(B_s^{(i)}-B_r^{(j)})dsdr\bigg\}\right].
\end{align}

\section{Appendix A: Local time}\label{local time}
	Consider the Brownian local time of a one dimensional Brownian motion $\{B_t^x\}_{t\in [0,T]}$ starting at $x\in\mathbb R$, at level $a\in\mathbb R$ and time $t\in [0,T]$:
\begin{align*}
	L_a^x(t)=\int_0^t\delta_a(B_s^x)ds,
\end{align*}
and notice that the latter can be seen as the usual Brownian local time $L_{a-x}(t)$.

It's known (e.g. the proof of proposition XIII-2.1. of \cite{revuz2013continuous}) that for a fixed $t$ the map $\mathbb R\ni a\mapsto L_a(t)$ is a.s. continuous and has compact support,  hence it follows that 
\begin{align*}
	\alpha_t=\int_{\mathbb R} |L_a(t)|^2 da< \infty, a.s.,
\end{align*}
this together with the invariance of Lebesgue measure implies that $a\mapsto L_a^x(t)$ belongs to $\mathbb \mathbb L^2(\mathbb R)$ almost surely.

Then the following Fourier-like series expansion holds a.s. 
\begin{align*}
	L_a^x(t)&=\sum_{j=1}^{\infty} \left(\int_{\mathbb R} L_y^x(t)e_j(y) dy\right) e_j(a)\\
	&=\sum_{j=1}^{\infty} \left(\int_0^t e_j(B_s^x)ds\right)e_j(a),
\end{align*}
where in the last equality we've used the occupation time formula.

By the Parseval's identity we have:
\begin{align}\label{local time parseval}
	\sum_{j=1}^{\infty} \left(\int_0^t e_j(B_s^x)ds\right)^2= \int_{-\infty}^{\infty}  |L_a^x(t)|^2 da= \int_{-\infty}^{\infty}  |L_{a}(t)|^2 da <\infty\; a.s.
\end{align}

\section{Appendix B: Second quantization and Conditional expectation}\label{SQ and CE}

Let $(\Omega,\mathcal A,\mathbb P^W)$ be a probability space then it's well know that if $X\in \mathbb L^2(\Omega,\mathcal A,\mathbb P^W)$ and $\mathcal G\subset\mathcal A$ is a sub-sigma-algebra, the conditional expectation $\mathbb E\left[X|\mathcal G\right]$
can be seen as the orthogonal projection of $X$ on $\mathbb L^2(\Omega,\mathcal G,\mathbb P^W)$.
In this appendix we will show an analogous property of the second quantization operator.

\begin{proposition}\label{SQ and CE prop}
	Let $A_K$ be the projection operator of definition \ref{projection} then the second quantization operator $\Gamma(A_K)$ coincides with the conditional expectation $\mathbb E[\bullet|\sigma(Z_1,...,Z_K)]$ where $\sigma(Z_1,...,Z_K)$ is the sigma algebra  generated by the family of i.i.d Gaussian random variables $(Z_1,...,Z_K)$.
\end{proposition}

\begin{proof}
We consider again the complete probability space $(\Omega,\mathfrak B,\mathbb P^W)$ treated in the introduction.
Let $X\in \mathbb L^2(\mathbb P^W)$ then a result by Cameron and Martin \cite{cameron1947orthogonal} tells us that $X$ has a series expansion of the form
\begin{align*}
	X&=\sum_{\alpha\in \mathcal J} x_{\alpha}\mathcal H_{\alpha},\;\text{convergence in } \mathbb L^2(\mathbb P^W)
\end{align*}
where $\mathcal J$ is the space of of all
sequences $\alpha=(\alpha_1,\alpha_2,...)$ with elements $\alpha_i\in\mathbb N_0$ and with compact support and
\begin{align*}
\mathcal H_{\alpha}:=\prod_{j=1}^{\infty}H_{\alpha_j}(Z_j),
\end{align*}
where $H_n(\bullet)$ is the $n$-th Hermite polynomial and $Z_j:=\int_{\mathbb R}e_j(x)dW_x$, are known as ``generalized Hermite polynomials'' or ``Wick polynomials'' (e.g. \cite{holden2009stochastic}) and 
\begin{align*}
x_{\alpha}=\left(\prod_{j=1}^{\infty}\alpha_j!\right)^{-1}\mathbb E\left[X\mathcal H_{\alpha}\right].
\end{align*}

Now lets take the conditional expectation of $X$ given the sigma algebra $\sigma_K:=\sigma(Z_1,...,Z_K)$.
It's well known that we are allowed to interchange conditional expectation with an $\mathbb L^2$ convergent series, yielding 
\begin{align*}
	\mathbb E\left[X|\sigma_K\right]&=\sum_{\alpha\in \mathcal J} x_{\alpha}E\left[\mathcal H_{\alpha}|\sigma_K\right]\\
	&=\sum_{\alpha\in\mathcal J} x_{\alpha} \mathbb E\left[\prod_{j=1}^{\infty} H_{\alpha_j}\left(Z_j\right)\bigg|\sigma_K\right],
\end{align*}
and at this point we notice that the terms of the product involving  $Z_j$ for $j\in \{1,2,...,N\}$ are $\sigma_K$-measurable and hence can be pulled outside the conditional expectation,
\begin{align*}
	\mathbb E[X|\sigma_K]=\sum_{\alpha\in\mathcal J} x_{\alpha} \prod_{i=1}^{K} H_{\alpha_i}\left(Z_i\right)\mathbb E\left[\prod_{j=K+1}^{\infty} H_{\alpha_j}\left(Z_j\right)\bigg|\sigma_K\right],
\end{align*}
all the remaining terms are independent from $\sigma_K$, and mutually independent which implies that
\begin{align*}
	\mathbb E[X|\sigma_K]=\sum_{\alpha\in\mathcal J} x_{\alpha} \prod_{i=1}^{K} H_{\alpha_i}\left(Z_i\right)\prod_{j=K+1}^{\infty}\mathbb E\left[ H_{\alpha_j}\left(Z_j\right)\bigg|\sigma_K\right].
\end{align*}

Furthermore since the Hermite polynomials of a centered Gaussian random variables can be seen as its Wick  power i.e. $H_n(Z_j)= Z_j^{\diamond n}$ (\cite{janson1997gaussian} Theorem $3.19$), and since $\mathbb{E}(Z_j^{\diamond n})=\mathbb{E}(Z_j)^n=0$ we see that the only non-vanishing terms are those corresponding to the $\alpha$'s containing only positive values in the first $K$ entries (remember that $H_0(\cdot)\equiv 1$). 
This allows us to conclude that 
\begin{align}
	\mathbb E[X|\sigma_K]=\sum_{\alpha\in\mathcal J^K} x_{\alpha} \mathcal H_{\alpha}
\end{align}
where $\mathcal J^K:=\{\alpha\in\mathcal J:\alpha_i=0, \forall i>K\}$.

On the other hand we could write the Chaos decomposition in terms of multiple Wiener integrals, i.e.
\begin{align*}
	X=\sum_{n=0}^{\infty} I_n(f_n),
\end{align*}
where the kernel $f_n$ is a symmetric function in $\mathbb \mathbb L^2(\mathbb R)$.
Then by definition of the second quantization operator we have
\begin{align*}
	\Gamma(A_K)X&=\sum_{n=0}^{\infty} I_n(\Gamma(A_K)^{\otimes n} f_n)\\
	&=\sum_{n=0}^{\infty} \sum_{ \alpha\in\mathcal J^K_n}\bigg\langle f_n,\bigodot_{j=1}^{K}e_j^{\odot \alpha_j}\bigg\rangle_{\mathbb \mathbb L^2(\mathbb R)^{\odot n}} I_n\left(	 \bigodot_{j=1}^{K}e_j^{\odot \alpha_j} \right)\\
	&=\sum_{n=0}^{\infty}\sum_{ \alpha\in\mathcal J^K_n} x_{\alpha} \mathcal H_{\alpha}\\
	&=\sum_{n=0}^{\infty}\sum_{\alpha\in\mathcal J^K_n} x_{\alpha} \mathcal H_{\alpha}\\
	&=\sum_{\alpha\in\mathcal J^K} x_{\alpha} \mathcal H_{\alpha}\\
	&=\mathbb E[X|\sigma_K],
\end{align*}

where we have used the following identity proved by It\^o \cite{ito1951multiple}
\begin{align*}
	\mathcal H_{\alpha}= I_n\left(	 \bigodot_{j=1}^{K}e_j^{\odot \alpha_j} \right),
\end{align*}

and for $\alpha\in\mathcal J_n^K:=\{\alpha\in\mathcal J:|\alpha|=n, \alpha_i=0, \forall i>K\}$ we let
\begin{align*}
x_{\alpha}=\bigg\langle f_n,\bigodot_{j=1}^{K}e_j^{\odot \alpha_j}\bigg\rangle_{\mathbb \mathbb L^2(\mathbb R)^{\odot n}}=n!\bigg\langle f_n,\bigodot_{j=1}^{K}e_j^{\odot \alpha_j}\bigg\rangle_{\mathbb \mathbb L^2(\mathbb R^n)}.
\end{align*}
\end{proof}

\printbibliography

\end{document}